\documentclass[a4paper,11pt,reqno]{amsart}
\usepackage[utf8]{inputenc}
\usepackage[T1]{fontenc}
\usepackage{latexsym,amsmath,amsfonts,amscd,amssymb,amsthm,mathrsfs}
\usepackage{pifont}
\usepackage{tabu}
\usepackage{eucal}
\usepackage{enumerate}
\usepackage[all]{xy}
\usepackage{color}
\usepackage{tikz}
\definecolor{cobalt}{RGB}{61,89,171}
\usepackage[colorlinks,citecolor=cobalt,linkcolor=cobalt,urlcolor=cobalt,pdfpagemode=UseNone,backref = page]{hyperref}

\theoremstyle{plain}
\newtheorem{theorem}{Theorem}[section]

\theoremstyle{definition}
\newtheorem{definition}[theorem]{Definition}

\theoremstyle{remark}

\begin{document}

\title[Bott-Chern and Aeppli cohomology]{On the Bott-Chern and Aeppli cohomology}

\author[D. Angella]{Daniele Angella}

\address[D. Angella]{Centro di Ricerca Matematica ``Ennio de Giorgi''\\
Collegio Puteano, Scuola Normale Superiore\\
Piazza dei Cavalieri 3\\
56126 Pisa, Italy
}

\email{daniele.angella@gmail.com}
\email{daniele.angella@sns.it}

% \urladdr{http://sites.google.com/site/danieleangella/}

\thanks{The author is supported by the Project PRIN ``Varietà reali e complesse: geometria, topologia e analisi armonica'', by the Project FIRB ``Geometria Differenziale e Teoria Geometrica delle Funzioni'', by SNS GR14 grant ``Geometry of non-Kähler manifolds'', and by GNSAGA of INdAM}

\dedicatory{Dedicated to the memory of Professor Pierre Dolbeault.}

\begin{abstract}
This survey summarizes the results discussed in a talk at ``Bielefeld Geometry \& Topology Days'' held at Bielefeld University in July 2015. We are interested in quantitative and qualitative properties of Bott-Chern cohomology.
We announce new results obtained in \cite{angella-tardini-1} jointly with Nicoletta Tardini.
%[D. Angella, N. Tardini, Quantitative and qualitative cohomological properties for non-K\"ahler manifolds, arXiv:1507.07108].)
In particular, we prove an upper bound of the dimensions of Bott-Chern cohomology in terms of Hodge numbers. We also introduce a notion, called Schweitzer qualitative property, which encodes the existence of a non-degenerate pairing in Bott-Chern cohomology, like the Poincar\'e duality for the de Rham cohomology. We show that this property characterizes the $\partial\overline\partial$-Lemma.
\end{abstract}

\maketitle

\date{}

\section*{Introduction}

We are aimed at decoding some of the information on the complex geometry of a compact complex manifold from the cohomologies associated to its double complex of differential forms.
In particular, we focus on the information contained in the Bott-Chern and Aeppli cohomologies. They provide, in a sense, a bridge between the holomorphic contents of the Dolbeault cohomology, and the topological contents of the de Rham cohomology. In this sense, it is expected that they both provide a better control on the holomorphic structure, (see, e.g., \cite{angella-tomassini-3} and Theorem \ref{thm:main-thm},) and furnish natural tools for treating geometric aspects, (see, e.g., \cite{tosatti-weinkove}). A possible link between these two settings would be provided by a proof of the conjecture that {\itshape compact complex manifolds satisfying the $\partial\overline\partial$-Lemma admit balanced metrics in the sense of Michelsohn}, see \cite[\S6]{popovici}.

\medskip

In this survey, we focus on quantitative properties of Bott-Chern and Aeppli cohomologies towards the study of their qualitative properties.

More precisely, as for the ``quantitative'' aspects, we recall a result proven by the author and A. Tomassini in \cite{angella-tomassini-3}. It states that, once fixed the topological structure, there is a lower bound on the dimension of the Bott-Chern cohomology in terms of the Betti numbers. See \cite[Theorem A]{angella-tomassini-3} or Theorem \ref{thm:frolicher-bc} for a precise statement. Moreover, the lower bound is attained if and only if the complex manifold satisfies the $\partial\overline\partial$-Lemma, (that is, the identity induces natural isomorphisms between Bott-Chern, Aeppli, Dolbeault cohomologies). We announce here an upper bound for the Bott-Chern cohomology in terms of Hodge numbers, obtained in joint work with Nicoletta Tardini: further details and complete proofs can be found in \cite{angella-tardini-1}. More precisely, in Theorem \ref{thm:upper-bound}, we prove that, on a compact complex manifold $X$ of complex dimension $n$, for any $k\in\mathbb{Z}$, there holds
 $$ \sum_{p+q=k} \dim_\mathbb{C} H^{p,q}_{A}(X) \;\leq\; (n+1) \, \left( \sum_{p+q=k} \dim_\mathbb{C} H^{p,q}_{\overline\partial}(X) + \sum_{p+q=k+1} \dim_\mathbb{C} H^{p,q}_{\overline\partial}(X) \right) \;. $$
Note that a topological upper bound is not possible.

As a consequence, we get that the difference $\sum_{p+q=k} \left( \dim_\mathbb{C} H^{p,q}_{BC}(X) - \dim_\mathbb{C} H^{p,q}_{A}(X) \right)$ is bounded from both above and below by Hodge numbers. Such a quantity yields another characterization of the $\partial\overline\partial$-Lemma: in Theorem \ref{thm:char-deldelbar-minus}, in joint work with Nicoletta Tardini \cite{angella-tardini-1}, we show that a compact complex manifold $X$ satisfies the $\partial\overline\partial$-Lemma if and only if
$$ \sum_{k\in\mathbb{Z}} \left| \sum_{p+q=k} \left( \dim_\mathbb{C} H^{p,q}_{BC}(X) - \dim_\mathbb{C} H^{p,q}_{A}(X) \right) \right| \;=\; 0 \;. $$
This has to be compared with the characterization in \cite[Theorem B]{angella-tomassini-3}, which uses instead the vanishing of the non-negative degrees $\Delta^k$ in \eqref{eq:non-kahler-degree}:
$$ \sum_{k\in\mathbb{Z}} \left( \sum_{p+q=k} \left( \dim_\mathbb{C} H^{p,q}_{BC}(X) + \dim_\mathbb{C} H^{p,q}_{A}(X) \right) - 2\, b_k \right) \;=\; 0 \;, $$
and with \cite[Corollary 4.14]{angella-tomassini-5}, which is deduced from interpretation of complex structures as generalized-complex structures in the sense of N. Hitchin.
Note also that, by the Schweitzer duality between Bott-Chern cohomology and Aeppli cohomology \cite[\S2.c]{schweitzer}, the condition in Theorem \ref{thm:char-deldelbar-minus} can be written just in terms of Bott-Chern cohomology as
$$ \sum_{k\in\mathbb{Z}} \left| \sum_{p+q=k} \dim_\mathbb{C} H^{p,q}_{BC}(X) - \sum_{p+q=2n-k} \dim_\mathbb{C} H^{p,q}_{BC}(X) \right| \;=\; 0 \;; $$
compare with the Poincar\'e duality and the Serre duality.

\medskip

Our final aim would be to attempt the study of ``qualitative'' aspects, namely, of the concrete realization of the algebra structure in Bott-Chern cohomology. Some previous attempts were done in \cite{angella-tomassini-6, tardini-tomassini}, with the final motivation of understanding whether a possible notion of formality {\itshape \`a la} Sullivan for Bott-Chern cohomology makes sense. Roughly speaking, a manifold would be ``formal with respect to Bott-Chern cohomology'' when the Bott-Chern cohomology functor can be made ``concrete'' by means of a zigzag of morphisms of bi-differential bi-graded algebras being quasi-isomorphisms with respect to Bott-Chern cohomology. This is the case when there is a suitable choice of the representatives having by themselves a structure of algebra. (A stronger request would be to ask for harmonic representatives with respect to some Hermitian metric.)
We introduce here a ``qualitative'' notion, motivated by the following observation.
From the geometric point of view, the algebra structure is mainly important in connection also with the Poincar\'e duality. But Hermitian duality does not preserve Bott-Chern cohomology. We will introduce a property, called {\em Schweitzer qualitative property}, from the work in \cite{schweitzer}, that implies in fact the validity of the $\partial\overline\partial$-Lemma, thanks to the above quantitative results, see \cite{angella-tardini-1}.

\bigskip

\noindent{\sl Acknowledgments.}
This note has been written for the Workshop ``Bielefeld Geometry \& Topology Days'' held at Bielefeld University on July 2nd--3rd, 2015. The author warmly thanks the organizer Giovanni Bazzoni for the kind invitation and hospitality. Thanks to him and to all the participants for the enthusiastic and fruitful environment they contributed to. We acknowledge many important discussions with Nicoletta Tardini and Adriano Tomassini on the subject: we thank also Michela Zedda, Giovanni Bazzoni, Oliver Goertsches for interesting conversations and useful suggestions. Thanks to Valentino Tosatti and Jim Stasheff for useful comments that improved the presentation of the paper. The new results that we announce here are obtained in joint work with Nicoletta Tardini, \cite{angella-tardini-1}.

\section{Bott-Chern and Aeppli cohomology}

We study the geometry of compact complex manifolds $X$ as encoded in their double complex of forms, namely, $\left(\wedge^{\bullet,\bullet}X,\partial,\overline\partial\right)$. This is an object of $\mathbf{bba}$, that is, the category of bi-differential $\mathbb{Z}^2$-graded algebras. Note that any compact complex manifold admits two natural structures: a real structure; and a non-degenerate pairing structure. These are encoded also in symmetries of the double complex: conjugation yields a symmetry around the bottom-left/top-right diagonal; the duality yields a symmetry around the bottom-right/top-left diagonal.

We keep in mind the case that the double complex is a direct sum of the following two indecomposable objects\footnote{it seems that this assumption can be assumed without loss of generality, see the answer by Greg Kuperberg in the MathOverflow discussion at \url{http://mathoverflow.net/questions/25723/}, where he quotes Mikhail Khovanov}:
\begin{itemize}
 \item {\em zigzags} of length $\ell+1$, where $\ell\in\mathbb{N}$ counts the number of arrows;
 $$ 
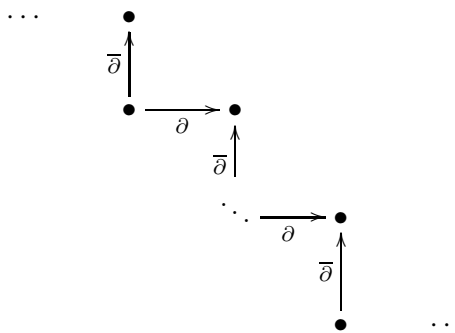
\xymatrix{
 \cdots & \bullet &&& \\
 & \bullet \ar[u]^{\overline\partial} \ar[r]_{\partial} & \bullet & & \\
 && \ddots \ar[u]^{\overline\partial} \ar[r]_{\partial} & \bullet & \\
 && & \bullet \ar[u]^{\overline\partial} & \cdots \\
 } $$
 
 \item {\em squares} of isomorphisms:
 
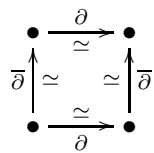
$$ \xymatrix{
 \bullet \ar[r]^{\partial}_{\simeq} & \bullet \\
 \bullet \ar[u]^{\overline\partial}_{\simeq} \ar[r]_{\partial}^{\simeq} & \bullet \ar[u]_{\overline\partial}^{\simeq}
 } $$
\end{itemize}
In particular, zigzags of length one are called {\em dots}. By Hodge theory and elliptic PDE theory, the cohomologies have finite dimension: whence the number of zigzag is finite, and the number of squares is infinite.

For example, the double complex associated to a hypothetical complex structure on the $6$-dimensional sphere $\mathbb{S}^6$ should be as in Figure \ref{fig:S6}. This example is constructed by using the results in \cite{ugarte-sphere} on the Fr\"olicher spectral sequence of a hypothetical complex structure on the six-sphere.

\begin{figure}
\begin{center}
\begin{tikzpicture}
\newcommand\un{2}

\draw[help lines, step=\un] (0,0) grid (4*\un,4*\un);

\foreach \x in {0,...,3}
  \node at (\un*.5+\un*\x,-.3) {\x};
\foreach \y in {0,...,3}
  \node at (-.3,\un*.5+\un*\y) {\y};

\coordinate (A) at (0*\un+1/2*\un, 0*\un+1/2*\un);
\coordinate (B) at (0*\un+1/4*\un, 1*\un+1/4*\un);
\coordinate (C) at (0*\un+1/2*\un, 1*\un+3/4*\un);
\coordinate (D) at (0*\un+1/4*\un, 2*\un+1/2*\un);
\coordinate (E) at (1*\un+1/2*\un, 0*\un+3/4*\un);
\coordinate (F) at (1*\un+1/4*\un, 0*\un+1/4*\un);
\coordinate (G) at (1*\un+1/4*\un, 1*\un+1/4*\un);
\coordinate (H) at (1*\un+3/4*\un, 1*\un+1/4*\un);
\coordinate (I) at (1*\un+1/2*\un, 1*\un+3/4*\un);
\coordinate (L) at (1*\un+1/4*\un, 2*\un+1/2*\un);
\coordinate (M) at (1*\un+3/4*\un, 2*\un+1/2*\un);
\coordinate (N) at (2*\un+1/2*\un, 0*\un+3/4*\un);
\coordinate (O) at (2*\un+1/4*\un, 0*\un+1/4*\un);
\coordinate (P) at (2*\un+3/4*\un, 1*\un+1/4*\un);
\coordinate (Q) at (2*\un+3/4*\un, 1*\un+3/4*\un);
\coordinate (R) at (2*\un+3/4*\un, 2*\un+1/2*\un);
\coordinate (S) at (3*\un+3/4*\un, 1*\un+3/4*\un);

\newcommand{\raggio}{1*\un pt}
\fill (A) circle (\raggio);
\fill (B) circle (\raggio);
\fill (C) circle (\raggio);
\fill (D) circle (\raggio);
\fill (E) circle (\raggio);
\fill (H) circle (\raggio);
\fill (I) circle (\raggio);
\fill (L) circle (\raggio);
\fill (M) circle (\raggio);
\fill (N) circle (\raggio);
\fill (O) circle (\raggio);
\fill (P) circle (\raggio);
\fill (S) circle (\raggio);

\draw (B) -- (G) -- (F) -- (O);
\draw (C) -- (I);
\draw (D) -- (L);
\draw (E) -- (N);
\draw (H) -- (P);
\draw (M) -- (R) -- (Q) -- (S);

\begingroup\makeatletter\def\f@size{6}\check@mathfonts
\node at (0*\un+1/2*\un, 0*\un+1/2*\un-.3) {$1$};
\node at (0*\un+1/4*\un, 1*\un+1/4*\un-.3) {$\alpha$};
\node at (0*\un+1/2*\un+.2, 1*\un+3/4*\un-.3) {$h^{0,2}+1-\alpha$};
\node at (0*\un+1/4*\un+.1, 2*\un+1/2*\un-.3) {$h^{0,2}-\beta$};
\node at (1*\un+1/2*\un+.1, 0*\un+3/4*\un-.3) {$h^{1,0}$};
\node at (1*\un+3/4*\un+.6, 1*\un+1/4*\un-.3) {$h^{1,1}-h^{0,2}+\alpha-1$};
\node at (1*\un+3/4*\un, 2*\un+1/2*\un-.3) {$\beta$};
\endgroup

\end{tikzpicture}
\end{center}
\label{fig:S6}
\caption{Main structure of the double complex associated to a hypothetical complex structure on the $6$-dimensional sphere.
The labels count the number of respective objects and $\alpha$, $h^{0,2}$, $\beta$, $h^{1,0}$, $h^{1,1}$ are unknown non-negative integers.
In this diagram, the infinite squares and the arrows arising from symmetries have been removed.}
\end{figure}
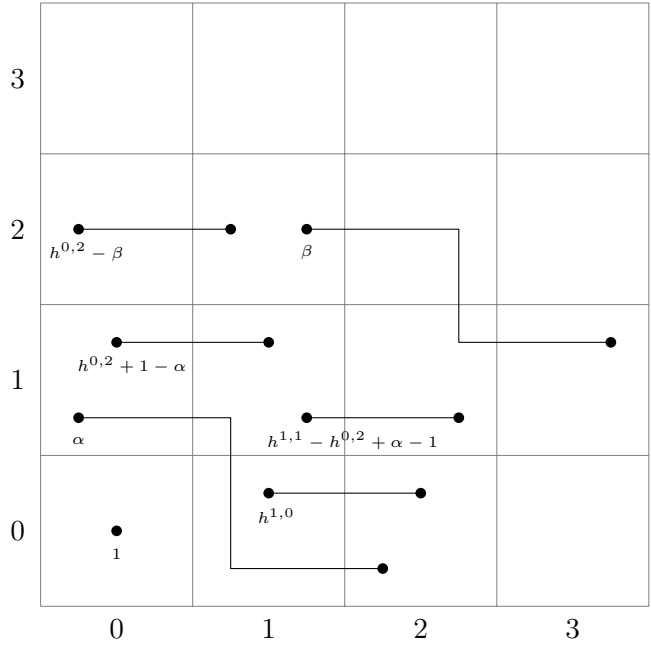

\medskip

Several cohomologies can be defined associated to a double complex.
The first ones that appear are the de Rham cohomology,
$$ H_{dR}(X;\mathbb{C}) \;=\; \frac{\ker (\partial+\overline\partial)}{\mathrm{imm} (\partial+\overline\partial)} \circ \mathrm{Tot} \;, $$
(here, $\mathrm{Tot}\colon \mathbf{bba}\to\mathbf{dga}$ denotes the totalization functor,
$$ \mathrm{Tot}(A^{\bullet,\bullet},\partial,\overline\partial) \;:=\; \left(\bigoplus_{p+q=\bullet}A^{p,q},\partial+\overline\partial\right) \;, $$
the category $\mathbf{dga}$ having differential $\mathbb{Z}$-graded algebras as objects,) and the Dolbeault cohomology and its conjugate,
$$ H_{\overline\partial}(X) \;=\; \frac{\ker \overline\partial}{\mathrm{imm} \overline\partial} \qquad \text{ and } \qquad H_{\partial}(X) \;=\; \frac{\ker \partial}{\mathrm{imm} \partial} \;. $$
In the model we keep in mind, constituted by squares and zigzags, the Dolbeault cohomology is easily computed by erasing vertical arrows with their ends from the diagram, and counting the remaining points.

Notice that Dolbeault and de Rham cohomology does not suffice, in general, for detecting the complete structure of the double complex. For example, zigzags of odd length do not contribute to the difference between Dolbeault and de Rham cohomology; in other words, to the higher terms in the Fr\"olicher spectral sequence. This means that symmetric zigzags of odd length cannot be detected. For example, the following diagrams have the same de Rham and Dolbeault cohomologies\footnote{this example was suggested by Michela Zedda}:

\begin{center}

\begin{minipage}{.45\textwidth}
\begin{center}
\begin{tikzpicture}

\draw[help lines] (0,0) grid (4,4);

\foreach \x in {0,...,3}
  \node at (.5+\x,-.3) {\x};
\foreach \y in {0,...,3}
  \node at (-.3,.5+\y) {\y};

\coordinate (a) at (1.5,2.5);
\coordinate (b) at (2.5,1.5);

\fill (a) circle (2.5pt);
\fill (b) circle (2.5pt);
  
\end{tikzpicture}
\end{center}
\end{minipage}
\begin{minipage}{.45\textwidth}
\begin{center}
\begin{tikzpicture}

\draw[help lines] (0,0) grid (4,4);

\foreach \x in {0,...,3}
  \node at (.5+\x,-.3) {\x};
\foreach \y in {0,...,3}
  \node at (-.3,.5+\y) {\y};

\coordinate (a1) at (1.5+.2,2.5+.2);
\coordinate (b1) at (2.5+.2,1.5+.2);
\coordinate (a2) at (1.5-.2,2.5-.2);
\coordinate (b2) at (2.5-.2,1.5-.2);
\coordinate (c) at (2.5+.2,2.5+.2);
\coordinate (d) at (1.5-.2,1.5-.2);

\fill (a1) circle (2.5pt);
\fill (b1) circle (2.5pt);
\fill (a2) circle (2.5pt);
\fill (b2) circle (2.5pt);
\fill (c) circle (2.5pt);
\fill (d) circle (2.5pt);

\draw (a1) -- (c);
\draw (b1) -- (c);
\draw (a2) -- (d);
\draw (b2) -- (d);

\end{tikzpicture}
\end{center}
\end{minipage}

\end{center}

The last discussion serves us as a motivation for introducing Bott-Chern and Aeppli cohomologies. Indeed, the above diagrams differ as to the number of corners; whence we get the need for having an invariant that counts the corners. In \cite{bott-chern, aeppli}, {\em Bott-Chern cohomology} and {\em Aeppli cohomology} are defined as
$$ H_{BC}^{\bullet,\bullet}(X) \;=\; \frac{\ker\partial\cap\ker\overline\partial}{\mathrm{imm}\partial\overline\partial}
\qquad \text{ and } \qquad
H_{A}^{\bullet,\bullet}(X) \;=\; \frac{\ker\partial\overline\partial}{\mathrm{imm}\partial+\mathrm{imm}\overline\partial} \;. $$
As for the Bott-Chern cohomology, it is easy to recognize that it counts the corners possibly having ingoing arrows, except for the squares\footnote{in the following diagram, the filled dot at the top-right is a generator of the Bott-Chern cohomology. Indeed, it is $\partial$-closed and $\overline\partial$-closed: this is pictured by the fact that there is no outgoing arrow. It is not $\partial\overline\partial$-exact, that is, it is not the top-right corner of a square. But it can be $\partial$-exact and/or $\overline\partial$-exact, whence the dotted arrows in the notation. Analogous considerations hold for the diagram for the Aeppli cohomology.}:
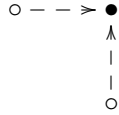
$$ \xymatrix{
\circ \ar@{-->}[r] & \bullet \\
&\circ \ar@{-->}[u]
}$$
Dually, the Aeppli cohomology counts the corners possibly having outgoing arrows, except for the squares:
$$ 
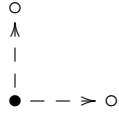
\xymatrix{
\circ  & \\
\bullet \ar@{-->}[r] \ar@{-->}[u] & \circ 
}$$

Finally, we have the following diagram, where the maps are morphisms of either $\mathbb{Z}$-graded or $\mathbb{Z}^2$-graded algebras naturally induced by the identity:
$$ \xymatrix{
  & H^{\bullet,\bullet}_{BC}(X) \ar[d]\ar[ld]\ar[rd] & \\
  H^{\bullet,\bullet}_{\partial}(X) \ar[rd] \ar@{=>}[r] & H^{\bullet}_{dR}(X;\mathbb{C}) \ar[d] & H^{\bullet,\bullet}_{\overline\partial}(X) \ar[ld] \ar@{=>}[l] \\
  & {\phantom{\;.}} H^{\bullet,\bullet}_{A}(X) \;. &
} $$
Here, the link between de Rham and (conjugate) Dolbeault cohomology is the Fr\"olicher spectral sequence, naturally arising from the structure of double-complex.

A compact complex manifold $X$ is said to satisfy the {\em $\partial\overline\partial$-Lemma} if the natural map $H^{\bullet,\bullet}_{BC}(X)\to H^{\bullet,\bullet}_{A}(X)$ induced by the identity is injective. This is equivalent to all the maps in the above diagram being isomorphisms, see \cite[Lemma 5.15]{deligne-griffiths-morgan-sullivan}. This is also equivalent to asking that the double complex associated to $X$ be a direct sum of squares and dots, \cite[Proposition 5.17]{deligne-griffiths-morgan-sullivan}.

\section{Quantitative properties of Bott-Chern cohomology}
In this section, we are interested in determining {\itshape quantitative} relations between the dimensions of the above cohomologies. This is ultimately related to what sequences of integers can appear as dimensions of cohomologies of double complexes of compact complex manifolds.
The known restrictions on such sequences are essentially:
\begin{itemize}
 \item restrictions arising from dimension, compactness, and connectedness;
 \item symmetries arising from the real structure and the non-degenerate pairing structure;
 \item inequalities of algebraic type as in \cite[Theorem 2]{frolicher}, \cite[Theorem A]{angella-tomassini-3}, and Theorem \ref{thm:upper-bound};
 \item inequalities of analytic type, as the ones that hold for compact complex surfaces, see \cite{angella-dloussky-tomassini} and subsequent work.
\end{itemize}

\medskip

A first result in this direction is the {\em Fr\"olicher inequality} in \cite[Theorem 2]{frolicher}. It states that the structure of double complex yields a spectral sequence of the form
$$ H^{\bullet,\bullet}_{\overline\partial} \Rightarrow H^\bullet_{dR}(X;\mathbb{C}) \;, $$
whence the inequality
$$ \text{for any }k\in\mathbb{Z}\;, \qquad h^{k}_{\overline\partial} - b_k \;\geq\; 0 \;. $$
(As a matter of notation, we write, e.g., $h^k_{\overline\partial}:=\sum_{p+q=k}\dim_{\mathbb{C}} H^{p,q}_{\overline\partial}(X)$; moreover, $b_k$ denotes the $k$th Betti number of $X$.)

The discussion in the previous paragraph motivates a similar inequality for the Bott-Chern cohomology. Indeed, recall that Dolbeault cohomology does not care horizontal arrows, conjugate Dolbeault cohomology does not care vertical arrows, Bott-Chern cohomology counts possibly incoming corners, Aeppli cohomology counts possibly outgoing corners, with the exception, in any case, of squares. Hence, just by combinatorial arguments, one recognizes that the sum of the dimension of the Bott-Chern and Aeppli cohomologies is greater than or equal than the sum of the dimension of Dolbeault and conjugate Dolbeault cohomologies, which is greater than or equal than twice the Betti number. Moreover, both equalities hold if and only if the double complex is direct sum of squares and dots, that is, if the manifold satisfies the $\partial\overline\partial$-Lemma. This heuristically explains the results in \cite{angella-tomassini-3}, to which we refer for details. The proof there uses the Varouchas exact sequences in \cite{varouchas}.

\begin{theorem}[{\cite[Theorem A, Theorem B]{angella-tomassini-3}}]\label{thm:frolicher-bc}
 Let $X$ be a compact complex manifold. Then, for any $k\in\mathbb{Z}$, the $k$th {\em non-$\partial\overline\partial$-degree}
 \begin{equation}\label{eq:non-kahler-degree}
 \Delta^k(X) \;:=\; h^{k}_{BC} + h^{k}_{A} - 2\, b_k
 \end{equation}
 is a non-negative integer. Moreover, $X$ satisfies the $\partial\overline\partial$-Lemma if and only if $\Delta^k(X)=0$ for any $k\in\mathbb{Z}$; equivalently,
 $$ \sum_{k\in\mathbb{Z}}\Delta^k(X) \;=\; 0 \;. $$
\end{theorem}

In a sense, the previous result states that an un-natural isomorphism (the one given by equality of the dimensions --- a quantitative property) ensures a natural isomorphism (the one induced by identity --- a qualitative property).

As a special case, consider compact complex surfaces.
In \cite{angella-dloussky-tomassini} and in further discussions with A. Tomassini and M. Verbitsky, we showed that the non-$\partial\overline\partial$-degree of compact complex surfaces are topological invariants. More precisely, $\Delta^1=0$ and $\Delta^2\in\{0,2\}$ according to the complex surface admitting K\"ahler metrics. (Notice that, for compact complex surfaces, the $\partial\overline\partial$-Lemma is in fact equivalent to the existence of a K\"ahler metric, by the Lamari and the Buchdahl criterion.)

\medskip

We announce now an upper bound for the dimension of Bott-Chern cohomology in terms of Hodge numbers, obtained in joint work with Nicoletta Tardini \cite{angella-tardini-1}. Note that we cannot get a just topological upper bound. That is, an opposite Fr\"olicher inequality cannot be obtained. This is because of the even-length zigzags, which contribute to the Dolbeault cohomology but not to the de Rham cohomology.

\begin{theorem}[{\cite[Theorem 2.1]{angella-tardini-1}}]\label{thm:upper-bound}
Let $X$ be a compact complex manifold of complex dimension $n$. Then, for any $k\in\mathbb{Z}$, there holds
 \begin{eqnarray*}
 h^k_{A} &\leq& \min\{k+1, (2n-k)+1\} \, \left( h^k_{\overline\partial} + h^{k+1}_{\overline\partial} \right) \\[5pt]
 &\leq& (n+1) \, \left( h^k_{\overline\partial} + h^{k+1}_{\overline\partial} \right) \;,
 \end{eqnarray*}
 whence also
 \begin{eqnarray*}
 h^k_{BC} &\leq& \min\{k+1, (2n-k)+1\} \, \left( h^k_{\overline\partial} + h^{k-1}_{\overline\partial} \right) \\[5pt]
 &\leq& (n+1) \, \left( h^k_{\overline\partial} + h^{k-1}_{\overline\partial} \right) \;.
 \end{eqnarray*}
\end{theorem}

\begin{proof}
 We give just the idea behind the proof. We refer to \cite{angella-tardini-1} for details. The point is that a contribution to Aeppli cohomology arises from zigzags of positive length $\ell+1$. Any such zigzag, when placed between total degrees $k$ and $k+1$, creates exactly two non trivial classes in either Dolbeault or conjugate Dolbeault cohomology at either degree $k$ or degree $k+1$, and at most $\lfloor\ell/2\rfloor+1$ classes in Aeppli cohomology at degree $k$. (In particular, $\lfloor\ell/2\rfloor+1\leq \min\{k+1,(2n-k)+1\}\leq n+1$.) The inequality for Bott-Chern cohomology follows from the Schweitzer duality \cite[\S2.c]{schweitzer} and by the Serre duality.
\end{proof}

\medskip

In particular, for any $k\in\mathbb{Z}$, there holds
$$ -2(n+1) \, \left( h^{k-1}_{\overline\partial} + h^{k}_{\overline\partial} \right) \;\leq\; h^k_{A}-h^{k}_{BC} \;\leq\; 2(n+1) \, \left( h^{k}_{\overline\partial} + h^{k+1}_{\overline\partial} \right) \;. $$
In the following result, from \cite{angella-tardini-1}, we give a characterization of the $\partial\overline\partial$-Lemma in terms of the above inequality.

\begin{theorem}[{\cite[Theorem 3.1]{angella-tardini-1}}]\label{thm:char-deldelbar-minus}
 A compact complex manifold $X$ satisfies the $\partial\overline\partial$-Lemma if and only if
 $$ \sum_{k\in\mathbb{Z}} \left| h^{k}_{BC} - h^{k}_{A} \right| \;=\; 0 \;. $$
\end{theorem}

\begin{proof}
 We give the idea of the proof, referring to \cite{angella-tardini-1} for details. Recall that the Bott-Chern cohomology counts the corners with possible incoming arrows, and the Aeppli cohomology counts the corners with possible outcoming arrows, with the exceptions of squares. Therefore the hypothesis can be restated as: for any anti-diagonal, the number of ingoing arrows equals the number of outgoing arrows, except for squares. Since no ingoing arrow can enter the anti-diagonal of total degree $0$, it follows that there is no zigzag of positive length in the whole diagram. That is, the $\partial\overline\partial$-Lemma holds.
\end{proof}

\section{Qualitative properties of Bott-Chern cohomology}

By investigating the {\itshape qualitative} properties of some cohomology, we mean the study of what and how algebraic structures are induced in cohomology from the space of forms.

\medskip

For example, let us focus on the differential graded algebra structure on the space of forms given by the wedge product and the exterior differential, and on the de Rham cohomology. By the Leibniz rule, it induces a structure of algebra in cohomology.
We look at $H_{dR}$ as a functor inside the category $\mathbf{dga}$ of differential $\mathbb{Z}$-graded algebras:
$$ H_{dR} \colon \mathbf{dga} \rightsquigarrow \mathbf{dga} \;. $$
We ask for what objects $X$ this functor can be made ``concrete'', that is, when it can be realized as a composition of quasi-isomorphisms and formal inverses of quasi-isomorphisms in $\mathbf{dga}$: e.g.,
$$
\xymatrix{
X \ar@{~>}[rrrrrrr] \ar[dr]_{\text{qis}} & & & & & & & H_{dR}(X) \\
 & C_1 & C_2 \ar[l]^{\text{qis}} & C_3 \ar[l]^{\text{qis}}\ar[r]_{\text{qis}} & \cdots \ar[r]_{\text{qis}} & C_{h-1} & C_h \ar[l]^{\text{qis}} \ar[ur]_{\text{qis}} & 
}
$$
By the existence of minimal models, see, e.g., \cite[Theorem in §II.3 at page 29]{wu}, \cite[Proposition 7.7]{bousfield-gugenheim}, this corresponds to the dga of forms and the dga of de Rham cohomology sharing the same model. A compact complex manifold whose double complex of forms satisfies such a property is called {\em formal} in the sense of Sullivan \cite{sullivan}. Note that the minimal model contains information on the rational homotopy groups of the manifold \cite{sullivan}: hence the rational homotopy type of formal manifolds is a formal consequence of their de Rham cohomology. Compact complex manifolds satisfying the $\partial\overline\partial$-Lemma (e.g., compact K\"ahler manifolds) are formal in the sense of Sullivan, \cite[Main Theorem]{deligne-griffiths-morgan-sullivan}. A theory of Dolbeault formality for complex manifolds has been developed in \cite{neisendorfer-taylor}.

Notice that every compact complex manifold is formal {\itshape in the category of $A_\infty$-algebras}\footnote{that is, strongly homotopy associative algebras: the category of $A_\infty$-algebras is equivalent to the category of differential graded co-algebras, by means of the bar construction. We refer to \cite{stasheff, keller, getzler-jones, dotsenko-shadrin-vallette} for definitions and details.}.
This follows from the Homotopy Transfer Principle by T.~V. Kadeishvili \cite{kadeishvili}. See \cite{merkulov, zhou, lu-palmieri-wu-zhang} for the explicit construction of the Merkulov model.
By \cite{lu-palmieri-wu-zhang}, the induced $A_\infty$-structure in cohomology yields the Massey products, up to sign. See also \cite{positselski}.

With these notations, a compact complex manifold is formal in the sense of Sullivan if there exists a system of representatives $H^\bullet$ for the cohomology such that the induced $A_\infty$-structure on $H^\bullet$ is actually an algebra structure.
A particular case is when the chosen representatives are actually the harmonic representatives with respect to some Hermitian metric. This last situation is referred as {\em geometric formality} in the sense of Kotschick \cite{kotschick}. In this case, a possible category to which one can restrict is the category whose objects are dgas with a non-degenerate pairing.

\medskip

Note that the wedge products on forms induces an algebra structure in Bott-Chern cohomology, and just an $H_{BC}$-module structure in Aeppli cohomology.
The duality pairing given by any fixed Hermitian metric is internal in de Rham and Dolbeault cohomologies by Poincar\'e and Serre dualities, and it yields an isomorphism between Bott-Chern and Aeppli cohomologies \cite[\S2.c]{schweitzer}.

This is an important issue, for example, in defining Massey products for Bott-Chern cohomology. In \cite{angella-tomassini-6}, triple Aeppli-Bott-Chern Massey products are defined, starting from Bott-Chern classes, and yielding a class in Aeppli cohomology, up to indeterminacy.

Hence, in \cite{angella-tardini-1}, we propose the following definition, which takes into consideration the structure of non-degenerate pairing.

\begin{definition}[{\cite[Definition 5.1]{angella-tardini-1}}]
 A compact complex manifold $X$ of complex dimension $n$ is said to satisfy the {\em Schweitzer qualitative property} if the natural pairing
 $$ H^{\bullet,\bullet}_{BC}(X) \times H^{\bullet,\bullet}_{BC}(X) \to \mathbb{C} \;, \qquad \left( [\alpha], [\beta] \right) \mapsto \int_X \alpha\wedge\beta $$
 induced by the wedge product and by the pairing with the fundamental class $[X]$ is non-degenerate.
\end{definition}

The above qualitative properties implies a quantitative property which, in turn, characterizes the qualitative property of the $\partial\overline\partial$-Lemma, thanks to Theorem \ref{thm:char-deldelbar-minus}. This is stated in \cite{angella-tardini-1}.

\begin{theorem}[{\cite[Theorem 5.2]{angella-tardini-1}}]\label{thm:main-thm}
 Let $X$ be a compact complex manifold. If it satisfies the Schweitzer qualitative property, then it satisfies the $\partial\overline\partial$-Lemma.
\end{theorem}


\begin{thebibliography}{48}

\bibitem{aeppli}
\textsc{A. Aeppli},
\emph{On the cohomology structure of Stein manifolds},
in {\em Proc. Conf. Complex Analysis (Minneapolis, Minn., 1964)}, Springer, Berlin, 1965, 58--70.

\bibitem{angella-dloussky-tomassini}
\textsc{D. Angella}, \textsc{G. Dloussky}, \textsc{A. Tomassini},
\emph{On Bott-Chern cohomology of compact complex surfaces},
to appear in {Ann. Mat. Pura Appl.}, DOI: 10.1007/s10231-014-0458-7.

\bibitem{angella-tardini-1}
\textsc{D. Angella}, \textsc{N. Tardini},
\emph{Quantitative and qualitative cohomological properties for non-K\"ahler manifolds},
\texttt{arXiv:1507.07108}.

\bibitem{angella-tomassini-3}
\textsc{D. Angella}, \textsc{A. Tomassini},
\emph{On the $\partial\overline{\partial}$-lemma and Bott-Chern cohomology},
{Invent. Math.} \textbf{192} (2013), no.~1, 71--81.

\bibitem{angella-tomassini-5}
\textsc{D. Angella}, \textsc{A. Tomassini},
\emph{Inequalities à la Frölicher and cohomological decompositions},
{J. Noncommut. Geom.} \textbf{9} (2015), no.~2, 505--542. 

\bibitem{angella-tomassini-6}
\textsc{D. Angella}, \textsc{A. Tomassini},
\emph{On Bott-Chern cohomology and formality},
{J. Geom. Phys.} \textbf{93} (2015), 52--61. 

\bibitem{bott-chern}
\textsc{R. Bott}, \textsc{S.~S. Chern},
\emph{Hermitian vector bundles and the equidistribution of the zeroes of their holomorphic sections},
{Acta Math.} \textbf{114} (1965), no.~1, 71--112.

\bibitem{bousfield-gugenheim}
\textsc{A.~K. Bousfield}, \textsc{V.~K.~A.~M. Gugenheim},
\emph{On PL de Rham theory and rational homotopy type},
{Mem. Amer. Math. Soc.} \textbf{8} (1976), no.~179.

\bibitem{deligne-griffiths-morgan-sullivan}
\textsc{P. Deligne}, \textsc{Ph.~A. Griffiths}, \textsc{J. Morgan}, \textsc{D.~P. Sullivan},
\emph{Real homotopy theory of K\"ahler manifolds},
{Invent. Math.} \textbf{29} (1975), no.~3, 245--274.

\bibitem{dotsenko-shadrin-vallette}
\textsc{V. Dotsenko}, \textsc{S. Shadrin}, \textsc{B. Vallette},
\emph{De Rham cohomology and homotopy Frobenius manifolds},
{J. Eur. Math. Soc. (JEMS)} \textbf{17} (2015), no.~3, 535--547.

\bibitem{frolicher}
\textsc{A. Fr\"olicher},
\emph{Relations between the cohomology groups of Dolbeault and topological invariants},
{Proc. Natl. Acad. Sci. USA} \textbf{41} (1955), no.~9, 641--644.

\bibitem{getzler-jones}
\textsc{E. Getzler}, \textsc{J.~D.~S. Jones},
\emph{$A_\infty$-algebras and the cyclic bar complex},
{Illinois J. Math.} \textbf{34} (1990), no.~2, 256--283.

\bibitem{kadeishvili}
\textsc{ T.~V. Kadeishvili},
\emph{The algebraic structure in the homology of an A($\infty$)-algebra}, 
{Soobshch. Akad. Nauk Gruzin. SSR} \textbf{108} (1982), no.~2, 249--252. 

\bibitem{keller}
\textsc{B. Keller},
\emph{Introduction to $A$-infinity algebras and modules},
{Homology Homotopy Appl.} \textbf{3} (2001), no.~1, 1--35.

\bibitem{kotschick}
\textsc{D. Kotschick},
\emph{On products of harmonic forms},
{Duke Math. J.} \textbf{107} (2001), no.~3, 521--531.

\bibitem{lu-palmieri-wu-zhang}
\textsc{D.-M. Lu}, \textsc{J.~H. Palmieri}, \textsc{Q.-S. Wu}, \textsc{J.~J. Zhang},
\emph{$A$-infinity structure on Ext-algebras},
{J. Pure Appl. Algebra} \textbf{213} (2009), no.~11, 2017--2037.

\bibitem{merkulov}
\textsc{S.~A. Merkulov},
\emph{Strong homotopy algebras of a K\"ahler manifold},
{Internat. Math. Res. Notices} \textbf{1999} (1999), no.~3, 153--164.

\bibitem{neisendorfer-taylor}
\textsc{J. Neisendorfer}, \textsc{L. Taylor},
\emph{Dolbeault homotopy theory}, 
{Trans. Amer. Math. Soc.} \textbf{245} (1978), 183--210.

\bibitem{popovici}
\textsc{D. Popovici},
\emph{Volume and Self-Intersection of Differences of Two Nef Classes},
\texttt{arXiv:1505.03457}.

\bibitem{positselski}
\textsc{L. Positselski},
\emph{Koszulity of cohomology = $K(\pi,1)$-ness + quasi-formality},
\texttt{arXiv:1507.04691}.

\bibitem{schweitzer}
\textsc{M. Schweitzer},
\emph{Autour de la cohomologie de Bott-Chern},
Pr\'epublication de l'Institut Fourier no.~703 (2007),
\texttt{arXiv:0709.3528}.

\bibitem{stasheff}
\textsc{J.~D. Stasheff},
\emph{Homotopy associativity of H-spaces. I, II},
{Trans. Amer. Math. Soc.} \textbf{108} (1963), 275--292; ibid. \textbf{108} (1963) 293--312.

\bibitem{sullivan}
\textsc{D. Sullivan},
\emph{Infinitesimal computations in topology},
{Inst. Hautes Études Sci. Publ. Math.} \textbf{47} (1977), no.~1, 269--331.

\bibitem{tardini-tomassini}
\textsc{N. Tardini}, \textsc{A. Tomassini},
\emph{On geometric Bott–Chern formality and deformations},
\texttt{arXiv:1502.03706}.

\bibitem{tosatti-weinkove}
\textsc{V. Tosatti}, \textsc{B. Weinkove},
\emph{The complex Monge-Amp\`ere equation on compact Hermitian manifolds},
{J. Amer. Math. Soc.} \textbf{23} (2010), no.~4, 1187--1195.

\bibitem{ugarte-sphere}
\textsc{L. Ugarte},
\emph{Hodge numbers of a hypothetical complex structure on the six sphere},
{Geom. Dedicata} \textbf{81} (2000), no.~1--3, 173--179.

\bibitem{varouchas}
\textsc{J. Varouchas},
\emph{Propriet\'es cohomologiques d'une classe de vari\'et\'es analytiques complexes compactes},
in {\em S\'eminaire d'analyse P. Lelong-P. Dolbeault-H. Skoda, ann\'ees 1983/1984}, Lecture Notes in Math., vol. \textbf{1198}, 233--243, Springer, Berlin (1986).

\bibitem{wu}
\textsc{W.-t. Wu},
{\em Rational homotopy type. A constructive study via the theory of the $I^*$-measure},
Lecture Notes in Mathematics, \textbf{1264}, Springer-Verlag, Berlin, 1987.

\bibitem{zhou}
\textsc{J. Zhou},
\emph{Hodge theory and $A_\infty$-structures on cohomology},
{Internat. Math. Res. Notices} \textbf{2000} (2000), no.~2, 71--78.

\end{thebibliography}
\end{document}